\documentclass[10pt, reqno]{amsart}
\usepackage{graphicx, amssymb, amsmath, amsthm}
\numberwithin{equation}{section}

\usepackage{color}

\usepackage{hyperref}

\newcommand{\R}{\mathbb{R}}

\newcommand{\Z}{\mathbb{Z}}

\newtheorem{theorem}{Theorem}[section]

\newtheorem{lemma}[theorem]{Lemma}
\newtheorem{coro}[theorem]{Corollary}

\theoremstyle{definition}

\newtheorem{remark}[theorem]{Remark}

\makeatletter
\newcommand{\Extend}[5]{\ext@arrow0099{\arrowfill@#1#2#3}{#4}{#5}}
\makeatother

\begin{document}
\title[Maximal estimates]{Maximal estimates for fractional Schr\"odinger equations in scaling critical magnetic fields}

\author{Haoran Wang}
\address{Department of Mathematics, Beijing Institute of Technology, Beijing 100081, China}
\email{wanghaoran@bit.edu.cn}

\author{Jiye Yuan}
\address{Department of Mathematics, Beijing Key Laboratory on Mathematical Characterization, Analysis, and Applications of Complex Information, Beijing Institute of Technology, Beijing 100081, China}
%\address{Department of Mathematics, Beijing Institute of Technology, Beijing 100081, China}
%\address{Beijing Key Laboratory on Mathematical Characterization, Analysis, and Applications of Complex Information, Beijing Institute of Technology, Beijing 102488, P. R. China}
\email{yuan\_jiye@bit.edu.cn}

\begin{abstract}
In this paper, we combine the argument of \cite{FZZ} and \cite{SSWZ} to prove the maximal estimates for fractional Schr\"odinger equations
$(i\partial_t+\mathcal{L}_{\mathbf{A}}^{\frac\alpha 2})u=0$ in the purely magnetic fields
which includes the Aharonov-Bohm fields. The proof is based on the cluster spectral measure estimates. In particular $\alpha=1$, the maximal estimate for wave equation is sharp up to the endpoint.
\end{abstract}

%\bigskip\bigskip
\maketitle

\begin{center}
 \begin{minipage}{100mm}
 \end{minipage}
 \end{center}

\section{Introduction and main results}
In this paper, we consider the Cauchy problem of fractional Schr\"odinger equation
\begin{align}\label{AB-eq}
\begin{cases}
   &(i\partial_t+\mathcal{L}_{\mathbf{A}}^{\frac\alpha2})u=0,\quad (t,x)\in\mathbb{R}\times\mathbb{R}^2\setminus\{0\}, \\
   &u(0,x)=f(x),
  \end{cases}
\end{align}
where $\alpha$ is a positive real number and the operator $\mathcal{L}_{\mathbf{A}}$ is the magnetic Schr\"odinger operator given by
\begin{equation*}
\mathcal{L}_{\mathbf{A}}=\Big(i\nabla+\frac{\mathbf{A}(\hat{x})}{|x|}\Big)^2,\quad x\in \R^2\setminus\{0\},\quad \hat{x}=\frac{x}{|x|}\in\mathbb{S}^1,
\end{equation*}
where $\mathbf{A}\in W^{1,\infty}(\mathbb{S}^1,\mathbb{R}^2)$ satisfies the transversality condition
\begin{equation}\label{eq:transversal}
{\mathbf{A}}(\hat{x})\cdot\hat{x}=0,
\qquad
\text{for all }x\in\R^2.
\end{equation}
In particular, a typical example of ${\bf A}$ is known as the {\it Aharonov-Bohm} potential
\begin{equation}\label{ab-potential}
{\bf A}(\hat{x})=\sigma\Big(-\frac{x_2}{|x|},\frac{x_1}{|x|}\Big),\quad \sigma\in\R,
\end{equation}
which was initially used to study one of the most interesting and intriguing effect of quantum physics by Aharonov and Bohm in \cite{AB}. This Aharonov-Bohm effect occurs when electrons propagate in a domain with a zero magnetic field but with a nonzero vector potential, see \cite{PT89} and the references therein.
In another physic effect, the fermionic charges can be non-integer multiples of the Higgs charges in another typical cosmic-string scenarios observed by Alford and Wilczek \cite{AW}.
As the flux is quantized with respect to the Higgs charge, this will lead to  a non-trivial Aharonov-Bohm scattering of these fermions. From the mathematical point,
the operator $\mathcal{L}_{\mathbf{A}}$ can be extended as a self-adjoint operator (see \cite{FFFP}) and
the fractional operator $\mathcal{L}_{\mathbf{A}}^{\frac\alpha2}$ can be defined from the spectral perspective, see~\cite{FLS}~ for the fractional operator.
The solution of ~\eqref{AB-eq}~ can be written as
\begin{equation}\label{AB-solution}
u(t,x)=e^{it\mathcal{L}_{\mathbf{A}}^{\frac\alpha2}}f(x),
\end{equation}
which includes the usual wave and Schr\"odinger equation as two special cases.
The dispersive and Strichartz estimates associated with  the magnetic Schr\"odinger operator $\mathcal{L}_{\mathbf{A}}$
have been extensively studied in \cite{FFFP,FZZ1,GYZZ} and we also refer to  \cite{FZZ,GWZZ} for the resolvent estimates.

In this paper, we consider the the minimal regularity of initial data for which the above solution \eqref{AB-solution} pointwisely converges to the initial data $f\in H_{\bf A}^s(\R^2)$,
\begin{equation}\label{point-A}
\lim_{t\to 0}e^{it\mathcal{L}_{\mathbf{A}}^{\frac\alpha2}}f(x)
=f(x),\quad a.\, e.
\end{equation}
that is, the pointwise convergence problem related to the operator $\mathcal{L}_{\mathbf{A}}$.
The problem for free Schr\"odinger $(i\partial_t-\Delta)u=0$ with initial data $f\in H^s$ was first proposed by Carleson in ~\cite{Carl}~ (see also ~\cite{KR}~).
In this classical case, the solution can be formally expressed by using Fourier transform as
\begin{equation}\label{ss}
u(t,x)=e^{it\Delta}f(x)=\int_{\mathbb{R}^n}e^{2\pi ix\cdot\xi}e^{-2\pi it|\xi|^2}\hat{f}(\xi)d\xi.
\end{equation}
The standard routine for the pointwise convergence of the free Schr\"odinger operator \eqref{ss} is to find the minimal regularity of $f$
for which the corresponding maximal estimate holds as follows
\begin{equation}\label{max-s}
  \left\|\sup_{0<t<1}|e^{it\Delta}f(x)|\right\|_{L^{q}(B^{n}(0,1))}\le C_{n,q,s}\|f\|_{H^{s}(\R^n)}
\end{equation}
where $B^{n}(0,1)$ is the unit ball in $\R^{n}$.

%In 1980, Carleson first put out a question of determining the minimal parameter $s>0$ such that for $\forall f\in H^{s}(\R)$, there holds
%\begin{equation}\label{point}
%  \lim_{t\rightarrow 0}e^{it\Delta}f(x) = f(x), \quad\text{for}\,\,a.e.\, x\in \R.
%\end{equation}
In one dimension, Carleson proved that \eqref{max-s} holds for $s\ge\frac{1}{4}$. In 1981, Dahlberg and Kenig \cite{D-K} proved that the result obtained by Carleson is sharp.
In higher dimension for $n\ge2$, Sj\"olin \cite{Sjolin} and Vega \cite{Vega} independently showed that \eqref{max-s} holds for $s>\frac{1}{2}$.
Later, this range was improved by Lee \cite{Lee} to $s>\frac{3}{8}$ for $n=2$, and by Bourgain \cite{Bour2013} to $s>\frac{1}{2}-\frac{1}{4n}$ for $n\ge 3$.
Bourgain also proved that $s\ge\frac{1}{2}-\frac{1}{n}$ is necessary for \eqref{max-s} to hold. Du, Guth and Li \cite{DGL} obtained $s>\frac{1}{3}$ in dimension $n=2$.
Recently, Du and Zhang \cite{DZ} proved that \eqref{max-s} holds for $s>\frac{n}{2(n+1)}$ for $n\ge3$. Luc$\grave{a}$ and Rogers obtained the necessary condition
$s\ge\frac{n}{2(n+2)}$ and later Bougain obtained the best necessary condition $s\ge\frac{n}{2(n+1)}$ up to now.\vspace{0.2cm}

%Similarly in this paper, the main task is to establish the maximal estimates for fractional Schr\"odinger equations associated with the operator $\mathcal{L}_{\mathbf{A}}$, which include the usual wave and Schr\"odinger equation as two special cases.

%We briefly review the history of related maximal estimates here. There are many researchers devoted to the study of maximal estimates for the solution associated with the free Schr\"odinger $\mathcal{L}_{0,0}$ without a potential, i.e., $\mathbf{A}\equiv0, a\equiv0$. In this classical case, the solution can be formally expressed by using Fourier transform as
%\begin{equation*}
%u(t,x)=e^{it\Delta}f(x)=\int_{\mathbb{R}^n}e^{2\pi ix\cdot\xi}e^{-2\pi it|\xi|^2}\hat{f}(\xi)d\xi.
%\end{equation*}
%This direction has been completed so far, as we all known, and we shall not list related references here because it is rather easy to find for interested readers.
%However, we would like to mention the interesting result ~\cite{BLN}~, in which the authors consider the Schr\"odinger equation with orthonormal initial data.
It is nature to ask the same problem for more general dispersive equations associated with Schr\"odinger operator with potentials.
However, due to the lack of Fourier transform, there are much fewer results compared with the free Schr\"odinger case.
The picture in this direction is far to be complete since many powerful tools (e.g. decoupling, polynomial decomposition) cannot be applied in this setting.
It worths to mention that one of the most interesting potentials, the so-called inverse-square potential, which prevents us from using Fourier transform to give the explicit expression of the solution and hence the related techniques of Fourier analysis fail to work.
However, by replacing the Fourier transform by the Hankel transform, the authors of ~\cite{ZZM}~ established the maximal estimates for Schr\"odinger equation with inverse-square potential, in which they write the solution in terms of spherical harmonics. For more details, we refer the interested readers to ~\cite{ZZM}~ and the references therein. As for other types of maximal estimates, one is referred to ~\cite{B,B1}~ for quadratic Weyl sum and ~\cite{KLO}~ for average over space curve.
Even though, almost all of the results are far to be sharp in the setting of equations with potentials.\vspace{0.2cm}

To the best of our knowledge, there are few papers on the maximal estimates related to the magnetic operator $\mathcal{L}_{\mathbf{A}}$ available, for which we start a new program to study the maximal estimates associated with the operator $\mathcal{L}_{\mathbf{A}}$. To be specific, we try to establish the local- and global-in-time maximal estimates associated with the operator $\mathcal{L}_{\mathbf{A}}$. The proof is based on the cut-off spectral measure estimates.
To state the main theorems, we define the distorted Sobolev space as follows
\begin{align}\label{AB-inhomo}
&H^s_{\mathbf{A}}(\mathbb{R}^2):=(I+\mathcal{L}_{\mathbf{A}})^{-\frac{s}{2}}L^2(\mathbb{R}^2).
\end{align}
In particular, we shall write
\begin{equation}\label{AB-inhomo-norm}
\|f\|_{H^s_\mathbf{A}(\mathbb{R}^2)}:=\|(I+\mathcal{L}_{\mathbf{A}})^{\frac{s}{2}}f\|_{L^2(\mathbb{R}^2)},
\end{equation}
in what follows.
Now we are in the position to state our main results.
\begin{theorem}[Local estimate]\label{thm:local}
Let $u(t,x)=e^{it\mathcal{L}_{\mathbf{A}}^{\frac\alpha2}}f(x),\alpha>0$ be the solution to ~\eqref{AB-eq}~ with $f\in H^s_\mathbf{A}(\R^2)$, then
for

$\bullet$ either $2\leq q\leq 6$ and $s>\frac\alpha4(\frac6q-1)+(1-\frac{2}{q})$,%2(\frac12-\frac1q)$,

$\bullet$ or $6\leq q<\infty$ and $s>1-\frac{2}{q}$,

there exists a constant $C$ such that
\begin{equation}
\begin{split}
\big\|\sup_{|t|\leq1}|e^{it\mathcal{L}_{\mathbf{A}}^{\frac\alpha2}}f|\big\|_{L^q(\R^2)}\leq C\|f\|_{H_{\mathbf{A}}^s(\mathbb{R}^2)}.
\end{split}
\end{equation}

\end{theorem}

As a direct consequence of Theorem \ref{thm:local} with $q=2$ for $0<\alpha\le\frac{4}{3}$ and $q=6$ for $\alpha>\frac{4}{3}$, %by a standard argument
we have the following pointwise convergence result.
\begin{coro}
Let $s>\frac{\alpha}{2}$ when $0<\alpha\le\frac{4}{3}$ or $s>\frac{2}{3}$ when $\alpha>\frac{4}{3}$, then there holds the pointwise convergence \eqref{point-A}
for $ f\in H^{s}_{\mathbf{A}}(\R^2)$.
\end{coro}

\begin{remark}
In particular $\alpha=1$, i.e., the half-wave operator, from Rogers and Villarroya in \cite[Theorem 1]{RV}, the result is sharp up to endpoint $s=\frac12$.
But the result is far to be sharp in the most interesting and challenging Schr\"odinger case, that is, $\alpha=2$. We will continue investigating the problem in future work.
\end{remark}

\begin{remark}
In the free case ${\bf A}\equiv0$, in which the Fourier transform is available, Sj\"olin and Walther independently used the stationary phase to obtain the corresponding maximal estimate for the operator $e^{it(-\Delta)^{\frac{\alpha}{2}}}$. In one dimension Sj\"olin \cite{Sjolin} obtained the pointwise convergence holds for $s\ge\frac{1}{4}$ for $\alpha>1$ which is sharp, and Walther \cite{W} proved that the pointwise convergence holds for $s>\frac{\alpha}{4}$ for $0<\alpha<1$, which is almost sharp up to the endpoint. In higher dimension, for $\alpha>1$ Sj\"olin \cite{Sjolin} obtained the result of $s>\frac{1}{2}$ in $\R^n$ if $n\ge3$
and $s\ge\frac{1}{2}$ if $n=2$, and the pointwise convergence for the free Schr\"odinger equation for $0<\alpha<1$ in $\R^n$ with $n\ge2$ is left to be solved.
\end{remark}

\begin{theorem}[Global estimate]\label{thm:global}
Let $u(t,x)=e^{it\mathcal{L}_{\mathbf{A}}^{\frac\alpha2}}f(x), \alpha>0$ be the solution to ~\eqref{AB-eq}~ with $f\in \dot{H}^s_\mathbf{A}(\mathbb{R}^2)$, then for $6\leq q<\infty$ and $s>1-\frac{2}{q}$, we have
\begin{equation}\label{est:global}
\begin{split}
\big\|\sup_{t\in\R}|e^{it\mathcal{L}_{\mathbf{A}}^{\frac\alpha2}}f|\big\|_{L^q(\R^2)}\leq C\|f\|_{\dot{H}_{\mathbf{A}}^s(\R^2)}.
\end{split}
\end{equation}

\end{theorem}

\begin{remark} These results extend the local maximal estimates in Theorem \ref{thm:local} to global ones when $q\geq6$.
In the case of $\alpha=1$, which corresponds to the wave equation with a magnetic field,
the global maximal estimate is sharp up to the endpoint regularity, we refer to~\cite[Theorem 2]{RV}.
In the Schr\"odinger case $\alpha=2$, from \cite[Theorem 8]{RV1}, the regularity assumption $s>1-\frac2q$ is necessary for \eqref{est:global} up to the endpoint $s=1-\frac2q$.
\end{remark}

We close this section by briefly stating the organization of the whole paper as follows.
In Sec.2, we provide the main tools needed in the proof of our main results, i.e., the spectral measure estimates.
In Sec.3 and 4, we prove Theorem \ref{thm:local} and Theorem \ref{thm:global} respectively.
\vspace{0.1cm}

\section{preliminaries}

In this section, we provide the key spectral cluster estimates, which play a critical role in our proof.
For $0<\epsilon\leq 1$, we define the spectral projector associated with the magnetic Schr\"odinger
operator $\mathcal{L}_{\mathbf{A}}$ on the frequencies $[k,k+\epsilon]$ by
\begin{equation}\label{chi-k-epsilon}
\chi_{k,\epsilon}(\mathcal{L}_{\mathbf{A}}^{\frac\alpha2})=\int_0^\infty \chi_{[k,k+\epsilon]}(\lambda^{\alpha}) \, dE_{\sqrt{\mathcal{L}_{\mathbf{A}}}}(\lambda),\qquad \alpha>0,
\end{equation}
where $\chi_E(s)$ is the characteristic function on the set $E$ and $dE_{\sqrt{\mathcal{L}_{\mathbf{A}}}}(\lambda)$ is the spectral measure given by
\begin{equation}\label{resolvent}
dE_{\sqrt{\mathcal{L}_{\mathbf{A}}}}(\lambda)=\frac{d}{d\lambda}E_{\sqrt{\mathcal{L}_{\mathbf{A}}}}(\lambda)d\lambda
=\frac{\lambda}{i\pi}\big(R(\lambda+i0)-R(\lambda-i0)\big)d\lambda.
\end{equation}
Here $R(\lambda\pm i0)$ is the (outcoming/ingoing) resolvent
$$R(\lambda\pm i0)=(\mathcal{L}_{\mathbf{A}}-(\lambda^2\pm i0))^{-1}=\lim_{\epsilon\searrow0}\big(\mathcal{L}_{\mathbf{A}}-(\lambda^2\pm i\epsilon)\big)^{-1}.$$
Very recently,  Fanelli-Zhang-Zheng \cite{FZZ} used the spectral measure constructed in \cite{GYZZ} to prove
the following resolvent estimates.

\begin{lemma}\cite[Theorem 3.1]{FZZ}\label{lem:2.1}
Let $1\leq p<\frac43$ and $4<q\leq \infty$ satisfy $\frac23\leq \frac1p-\frac1q<1$.
% that is, $(\frac1p,\frac1q)$ is contained in the pentagon (i.e. Figure \ref{figure1}).
Then there exists a constant $C=C(p,q)>0$ such that
\begin{equation}\label{resolvent-estimate}
\|(\mathcal{L}_{\mathbf{A}}-(\lambda^2\pm i0))^{-1}f\|_{L^q(\mathbb{R}^2)}\leq C\lambda^{ 2(\frac1p-\frac1q)-2}\|f\|_{L^p(\mathbb{R}^2)}, \quad
\end{equation}
where $(\mathcal{L}_{\mathbf{A}}-(\lambda^2\pm i0))^{-1}:=\lim_{\epsilon\to0^+}(\mathcal{L}_{\mathbf{A}}-(\lambda^2\pm i\epsilon))^{-1}$.

In particular, if $6\leq q<\infty$, then \eqref{resolvent-estimate} holds for $q=p'$.
\end{lemma}
Therefore, from \eqref{resolvent} and \eqref{resolvent-estimate}, we have the following spectral measure estimate
\begin{equation}
 \| dE_{\sqrt{\mathcal{L}_{\mathbf{A}}}}(\lambda)\|_{L^{p}(\R^2)\to L^q(\R^2)}\leq C\lambda^{ 2(\frac1p-\frac1q)-1}.
 \end{equation}

Integrating on a frequency band of width $\epsilon$, we obtain

\begin{lemma}\label{lem:2.2}
For the operator $\chi_{k,\epsilon}(\mathcal{L}_{\mathbf{A}}^{\frac\alpha2})$ defined by \eqref{chi-k-epsilon}, we have
\begin{equation}\label{spec-1}
\|\chi_{k,\epsilon}(\mathcal{L}_{\mathbf{A}}^{\frac\alpha2})\|_{L^{p}(\R^2)\to L^q(\R^2)} \leq C \big((k+\epsilon)^{\frac1\alpha}-k^{\frac1\alpha}\big) (k+\epsilon)^{ (2(\frac1p-\frac1q)-1)/\alpha},
\end{equation}
where $1\leq p<\frac43$ and $4<q\leq \infty$ satisfy $\frac23\leq \frac1p-\frac1q<1$.

In particular, for $6\leq q=p'<+\infty$, then
\begin{equation}\label{spec-2}
\|\chi_{k,\epsilon}(\mathcal{L}_{\mathbf{A}}^{\frac\alpha 2})\|_{L^{p}(\R^2)\to L^{p'}(\R^2)} \leq C \big((k+\epsilon)^{\frac1\alpha}-k^{\frac1\alpha}\big) (k+\epsilon)^{ (2(\frac1p-\frac1{p'})-1)/\alpha}.
\end{equation}

\end{lemma}
To prove the local result in Theorem \eqref{thm:local} for the range $2\leq q\leq6$, here we prove the following lemma.
\begin{lemma}\label{lem:L2-est}
Let $U(t)=e^{it\mathcal{L}_{\mathbf{A}}^{\frac\alpha2}}$, then there exists a constant $C$ such that $$\| D_t^{s}U(t)\|_{\dot H_{\bf A}^{\alpha s}\to L^2}\leq C,$$
where $0\leq s\leq 1$.
\end{lemma}

\begin{proof} By the interpolation, we only need to consider $s=0$ and $s=1$. By the spectral theory, we have
\begin{equation*}
T^s(t):=D_t^{s} U(t)=\int_0^\infty e^{it\lambda^\alpha} \lambda^{\alpha s}\,dE_{\sqrt{\mathcal{L}_{\mathbf{A}}}}(\lambda).
\end{equation*}
For $\varphi\in C_c^\infty([1,2])$, we define
\begin{equation}
T^s_k(t)=\int_0^\infty e^{it\lambda^\alpha}\varphi(2^{-k}\lambda)\lambda^{\alpha s} \,dE_{\sqrt{\mathcal{L}_{\mathbf{A}}}}(\lambda),
\end{equation}
then $T^s(t)=\sum_{k\in\Z} T^s_k(t)$. By orthogonality, we have
\begin{equation*}
\|T^s(t)f\|^2_{L^2(\R^2)}=\sum_{k\in\Z}\|T^s_k(t)f\|^2_{L^2(\R^2)}=\sum_{k\in\Z}\|T^s_k(t) f_k\|^2_{L^2(\R^2)},
\end{equation*}
where $f_k=\tilde{\varphi}(2^{-k}\sqrt{\mathcal{L}_{\mathbf{A}}}) f$ with $\tilde{\varphi}\in C_c^\infty([\frac12,4])$ such that %$\tilde{\varphi}\varphi=1$.
$\tilde{\varphi}=1$ on supp\,$\varphi$.
It thus suffices to prove
$$\|T^s_k(t)\|_{L^2\to L^2}\leq C2^{\alpha s k},$$
which is equivalent to $$\|\big(T^s_k(t)\big)^*T^s_k(t)\|_{L^2\to L^2}\leq C^22^{2\alpha s k}.$$
To this end, we set $\psi=\varphi^2 \in C_c^\infty([1, 2])$, then it is further enough to show
\begin{equation}\label{est:L2-ker}
\begin{split}
\Big|\int^\infty_0 \psi(2^{-k}\lambda)\lambda^{2\alpha s} dE_{\sqrt{\mathcal{L}_{\bf A}}}(\lambda;x,y)\Big|\lesssim \frac{2^{2k}2^{2\alpha s k}}{(1+2^k|x-y|)^{K}},\quad k\in\Z,\quad\forall K\geq0,
\end{split}
\end{equation}
where $C$ is a constant independent of $k\in \Z$.
On the one hand, by scaling,
\begin{equation}\label{scaling-L2}
\begin{split}
\Big|\int^\infty_0 \psi(2^{-k}\lambda)\lambda^{2\alpha s} &dE_{\sqrt{\mathcal{L}_{\bf A}}}(\lambda;x,y)\Big|\\
&=2^{2k}2^{2\alpha s k}\Big|\int^\infty_0 \psi(\lambda)\lambda^{2\alpha s} dE_{\sqrt{\mathcal{L}_{\bf A}}}(\lambda; 2^kx, 2^ky)\Big|.
\end{split}
\end{equation}
On the other,  by \cite[Proposition 1.1]{GYZZ} and integration by parts K times, we have
\begin{equation*}
\begin{split}
\Big|\int^\infty_0 \psi(\lambda) \lambda^{2\alpha s}dE_{\sqrt{\mathcal{L}_{\bf A}}}(\lambda;x,y)\Big|\lesssim  \frac{1}{(1+|x-y|)^{K}}+\int_0^\infty \frac{1}{(1+|d_s|)^{K}} |B_{\alpha}(s,\theta_1,\theta_2)| ds,
\end{split}
\end{equation*}
where for the definition of $d_s$ and $B_{\alpha}(s,\theta_1,\theta_2)$ we can refer to \cite{GYZZ}.
From \cite[(1.25)]{GYZZ}, one has $|d_s|^2\geq (r_1+r_2)^2\geq |x-y|^2$.
Hence, it follows from \cite[(3.3)]{GYZZ} that
\begin{equation*}
\begin{split}
\Big|\int^\infty_0 \psi(\lambda) \lambda^{2\alpha s} dE_{\sqrt{\mathcal{L}_{\bf A}}}(\lambda;x,y)\Big|\lesssim  \frac{1}{(1+|x-y|)^{K}}.
\end{split}
\end{equation*}
This together with \eqref{scaling-L2} yields \eqref{est:L2-ker}.

\end{proof}

\section{The proof of Theorem \ref{thm:local}}

In this section, we shall prove the local maximal estimate associated with the magnetic operator $\mathcal{L}_{\mathbf{A}}$.
We need to prove the following local-in-time estimate
\begin{equation}\label{local-1}
\begin{split}
\big\|e^{it\mathcal{L}_{\mathbf{A}}^{\frac\alpha2}}f\big\|_{L^q(\mathbb{R}^2;L^\infty([-1,1]))}\lesssim\|f\|_{H_{\mathbf{A}}^s(\mathbb{R}^2)}.
\end{split}
\end{equation}
To this end, we first observe that there exists a Schwarz function $\rho\in\mathcal{S}(\mathbb{R})$ with Fourier support $\text{supp}\,\hat{\rho}\subset(-2,2)$ such that
\begin{equation}\label{local-global}
\big\|e^{it\mathcal{L}_{\mathbf{A}}^{\frac\alpha2}}f\big\|_{L^q(\mathbb{R}^2;L^\infty([-1,1]))}\lesssim\big\|\rho(t) e^{it\mathcal{L}_{\mathbf{A}}^{\frac\alpha2}}f\big\|_{L^q(\mathbb{R}^2;L^\infty(\mathbb{R}))}.
\end{equation}
Then to prove ~\eqref{local-1}~, it is sufficient to show
\begin{equation}\label{global-1}
\big\|\rho(t) e^{it\mathcal{L}_{\mathbf{A}}^{\frac\alpha2}}f\big\|_{L^q(\mathbb{R}^2;L^\infty(\mathbb{R}))}\lesssim\|(I+\mathcal{L}_{\mathbf{A}})^{\frac{s}{2}}f\|_{L^2(\mathbb{R}^2)}.
\end{equation}
In the case that $2\leq q\leq 6$, by interpolation, it suffices to prove
\begin{equation}\label{q=2}
\begin{split}
\big\|\rho(t)e^{it\mathcal{L}_{\mathbf{A}}^{\frac\alpha2}}f\big\|_{L^2(\mathbb{R}^2;L^\infty(\R))}\lesssim\|f\|_{H_{\mathbf{A}}^{\sigma\alpha}(\mathbb{R}^2)},\quad \sigma>\frac{1}{2}
\end{split}
\end{equation}
and
\begin{equation}\label{q=6}
\begin{split}
\big\|\rho(t) e^{it\mathcal{L}_{\mathbf{A}}^{\frac\alpha2}}f\big\|_{L^6(\mathbb{R}^2;L^\infty(\R))}
\lesssim\|f\|_{H_{\mathbf{A}}^{\sigma}(\mathbb{R}^2)},\quad \sigma>\frac{2}{3}.%\tfrac{2}{3}.
\end{split}
\end{equation}
The inequality \eqref{q=2} is a consequence of Lemma \ref{lem:L2-est} and the Sobolev embedding $\dot{H}_t^{\frac{1}{2}-\epsilon}(\R)\cap \dot{H}_t^{\frac{1}{2}+\epsilon}(\R)\hookrightarrow L^\infty_t(\R)$ with $0<\epsilon\ll1$.
Indeed, for $\sigma>\frac{1}{2}$, by Lemma \ref{lem:L2-est} we have
\begin{align*}
  \big\|\rho(t)e^{it\mathcal{L}_{\mathbf{A}}^{\frac\alpha2}}f\big\|_{L^2(\mathbb{R}^2;L^\infty(\R))}
  &=\big\|e^{it\mathcal{L}_{\mathbf{A}}^{\frac\alpha2}}(\rho(t)f)\big\|_{L^2(\mathbb{R}^2;L^\infty(\R))}\\
  &\lesssim\big\|D_{t}^{\sigma}e^{it\mathcal{L}_{\mathbf{A}}^{\frac\alpha2}}(\rho(t)f)\big\|_{L^2(\mathbb{R}^2;L^{2}(\R))}\\
  &\lesssim\big\|\rho(t)\big\|_{L^{2}(\R)}\|f\|_{H_{\mathbf{A}}^{\sigma\alpha}(\mathbb{R}^2)}\\
  &\lesssim\|f\|_{H_{\mathbf{A}}^{\sigma\alpha}(\mathbb{R}^2)}.
\end{align*}
The inequality \eqref{q=6} is the endpoint of the following case $q\geq 6$. Therefore, from now on, we focus on the case that $6\leq q<+\infty$.

Recall \eqref{chi-k-epsilon} with $\epsilon=1$
\begin{equation}\label{k-alpha}
%\chi_{k}(\mathcal{L}_{\mathbf{A}}^{\frac\alpha2})
\chi_{k}^{\alpha}:=\chi_{k,1}(\mathcal{L}_{\mathbf{A}}^{\frac\alpha2})=\int_0^\infty \chi_{[k,k+1]}(\lambda^{\alpha}) \, dE_{\sqrt{\mathcal{L}_{\mathbf{A}}}}(\lambda),
\end{equation}
then we have the decomposition $f=\sum_{k=0}^\infty\chi_{k}^{\alpha}f$.
By \eqref{spec-2}, for $6\leq q<\infty$, we have
\begin{equation}\label{chi-k-q}
\|\chi_{k}^{\alpha}f\|_{L^q(\mathbb{R}^2)}\lesssim \big((k+1)^{\frac1\alpha}-k^{\frac1\alpha}\big)^{\frac12}(k+1)^{(\frac12-\frac2q)/\alpha}\|f\|_{L^2(\mathbb{R}^2)}, \quad \forall k\geq0.
\end{equation}
Let $0<\delta\ll1$ and
\begin{equation}\label{F}
F(t,x):=|D_t|^{1/2+\delta}(\rho(t)e^{it\mathcal{L}_{\mathbf{A}}^{\frac\alpha2}}f(x)),
\end{equation}
then by Sobolev embedding, we obtain
\begin{equation}\label{rho-F}
\big\|\rho(t) e^{it\mathcal{L}_{\mathbf{A}}^{\frac\alpha2}}f\big\|_{L^q(\mathbb{R}^2;L^\infty(\mathbb{R}))}\lesssim\|F(t,x)\|_{L^q(\mathbb{R}^2;L^2(\mathbb{R}))}.
\end{equation}
Set
\begin{equation}\label{F-k}
F_k(t,x):=|D_t|^{1/2+\delta}(\rho(t)e^{it\mathcal{L}_{\mathbf{A}}^{\frac\alpha2}}\chi_{k}^{\alpha}f(x)),
\end{equation}
then
\begin{equation}\label{F-F-k}
F(t,x)=\sum_{k=0}^\infty F_k(t,x).
\end{equation}
Taking Fourier transform in variable $t$, we get
\begin{equation}\label{Four-F}
\hat{F}_k(\tau,x)=|\tau|^{1/2+\delta}\int_0^\infty\hat{\rho}(\tau-\lambda^\alpha)\chi_{[k,k+1]}(\lambda^\alpha)dE(\lambda)f(x).
\end{equation}
By Plancherel's identity and the compact support property of $\hat{\rho}$, we have
\begin{equation}\label{orth-F}
\int_{\mathbb{R}}F_k(t,x)\overline{F_\ell(t,x)}dt=(2\pi)^{-1}\int_{\mathbb{R}}\hat{F}_k(\tau,x)\overline{\hat{F}_\ell(\tau,x)}d\tau=0,\quad |k-\ell|>100.
\end{equation}
Therefore we get
\begin{equation}\label{orth-L2}
\|F(t,x)\|_{L_t^2(\mathbb{R})}^2\lesssim\sum_{k=0}^\infty\|F_k(t,x)\|_{L_t^2(\mathbb{R})}^2=(2\pi)^{-1}\sum_{k=0}^\infty\|\hat{F}_k(\tau,x)\|_{L_\tau^2(\mathbb{R})}^2.
\end{equation}
This together with \eqref{rho-F} gives
\begin{equation}\label{orth-rho}
\big\|\rho(t) e^{it\mathcal{L}_{\mathbf{A}}^{\frac\alpha2}}f\big\|_{L^q(\mathbb{R}^2;L^\infty(\mathbb{R}))}\lesssim\big\|\big(\sum_{k=0}^\infty
\|\hat{F}_k(\tau,x)\|_{L_\tau^2(\mathbb{R})}^2\big)^{1/2}\big\|_{L^q(\mathbb{R}^2)}.
\end{equation}
Notice that $q\geq6$ and \eqref{Four-F}, we obtain
\begin{align}\label{Minkowski}
&\big\|\rho(t) e^{it\mathcal{L}_{\mathbf{A}}^{\frac\alpha2}}f\big\|^2_{L^q(\mathbb{R}^2;L^\infty(\mathbb{R}))}
\lesssim\sum_{k=0}^\infty\int_{\mathbb{R}}\|\hat{F}_k(\tau,x)\|^2_{L^q(\mathbb{R}^2)}d\tau\\
&=\sum_{k=0}^\infty\int_{\mathbb{R}}|\tau|^{1+2\delta}\Big\|\int_0^\infty\hat{\rho}(\tau-\lambda^\alpha)\chi_{[k,k+1]}
(\lambda^\alpha)dE(\lambda)f(x)\Big\|^2_{L^q(\mathbb{R}^2)}d\tau.\nonumber
\end{align}
By \eqref{chi-k-q}, we thus get
\begin{align*}
&\big\|\rho(t) e^{it\mathcal{L}_{\mathbf{A}}^{\frac\alpha2}}f\big\|^2_{L^q(\mathbb{R}^2;L^\infty(\mathbb{R}))} \\
&\lesssim\sum_{k=0}^\infty\int_{k-3}^{k+3}|\tau|^{1+2\delta}
\big((k+1)^{\frac1\alpha}-k^{\frac1\alpha}\big)(k+1)^{(1-\frac4q)/\alpha}\|\chi_{k}^{\alpha}f\|^2_{L^2(\mathbb{R}^2)}d\tau\nonumber\\
&\lesssim\sum_{k=0}^\infty\big((k+1)^{\frac1\alpha}-k^{\frac1\alpha}\big)(k+1)^{\frac1\alpha(1-\frac4q)+1+2\delta}\|\chi_{k}^{\alpha}f\|^2_{L^2(\mathbb{R}^2)}\nonumber\\
&=\sum_{k=0}^\infty\|(1+k)^{\frac1\alpha(1-\frac2q)+\delta}\chi_{k}^{\alpha}f\|^2_{L^2(\mathbb{R}^2)}
\approx\|(I+\sqrt{\mathcal{L}_{\mathbf{A}}})^{1-\frac2q+\delta\alpha}f\|^2_{L^2(\mathbb{R}^2)}.\nonumber
\end{align*}

\section{Proof of theorem \ref{thm:global}}

In this section, we extend the local estimates in Theorem \ref{thm:local} to global estimates.  In fact, we shall prove the following estimate:
\begin{equation}\label{global-max}
\|u\|_{L^q(\mathbb{R}^2;L^\infty(\mathbb{R}))}\lesssim\|\mathcal{L}_{\mathbf{A}}^{s/2}f\|_{L^2(\mathbb{R}^2)},\quad6\leq q\leq\infty.
\end{equation}

To prove ~\eqref{global-max}~, we only need to show that, for $\epsilon\in(0,1)$, there exists a constant independent of $\epsilon$ such that
\begin{equation}
\|e^{it\mathcal{L}_{\mathbf{A}}^{\frac{\alpha}{2}}}f\|_{L^q(\mathbb{R}^2;L^\infty((-1/\epsilon,1/\epsilon)))}
\lesssim\|(\sqrt{\mathcal{L}_{\mathbf{A}}}+\epsilon I)^sf\|_{L^2(\mathbb{R}^2)}.
\end{equation}
Similar to the proof of Theorem \ref{thm:local}, it is further sufficient to show that
\begin{equation}
\|\rho(\epsilon t)e^{it\mathcal{L}_{\mathbf{A}}^{\frac{\alpha}{2}}}f\|_{L^q(\mathbb{R}^2;L^\infty(\mathbb{R}))}
\lesssim\|(\mathcal{L}_{\mathbf{A}}+\epsilon I)^{s/2}f\|_{L^2(\mathbb{R}^2)},
\end{equation}
where $\rho\in\mathcal{S}(\mathbb{R})$ is some fixed Schwarz function with supp$\hat{\rho}\subset(-2,2)$.

%{\bf Case of $\alpha\leq1$}:

Let
\begin{equation}
\chi_{k\epsilon}^{\alpha}:=\chi_{k\epsilon,\epsilon}(\mathcal{L}_{\mathbf{A}}^{\frac\alpha2})=\int_0^\infty \chi_{[k\epsilon,(k+1)\epsilon]}(\lambda^{\alpha}) \, dE_{\sqrt{\mathcal{L}_{\mathbf{A}}}}(\lambda),\quad \epsilon\in(0,1),
\end{equation}
then we have decomposition $f=\sum_{k=0}^\infty\chi_{k\epsilon}^{\alpha}f$ for a fixed value of $\epsilon$.

By \eqref{spec-2} and $TT^*$ argument, we have for $6\leq q<\infty$
\begin{equation}\label{chi-k-epsilon-q}
\|\chi_{k\epsilon}^{\alpha}f\|^2_{L^q(\mathbb{R}^2)}
\lesssim(((k+1)\epsilon)^{\frac1\alpha}-(k\epsilon)^{\frac1\alpha})((k+1)\epsilon)^{(1-\frac4q)/\alpha}\|f\|^2_{L^2(\mathbb{R}^2)}, \quad \forall k\geq0.
\end{equation}
Sobolev embedding $\dot{H}_t^{\frac{1}{2}+\epsilon}(\mathbb{R})\hookrightarrow L^\infty_t(\mathbb{R})$ will give us
\begin{equation}
\|\rho(\epsilon t)e^{it\mathcal{L}_{\mathbf{A}}^{\frac{\alpha}{2}}}f\|_{L^q(\mathbb{R}^2;L^\infty(\mathbb{R}))}
\lesssim\||D_t|^{1/2+\delta}(\rho(\epsilon t)e^{it\mathcal{L}_{\mathbf{A}}^{\frac{\alpha}{2}}}f)\|_{L^q(\mathbb{R}^2;L^2(\mathbb{R}))}.
\end{equation}
Let
\begin{equation}
F^\epsilon(t,x):=|D_t|^{1/2+\delta}(\rho(\epsilon t)e^{it\mathcal{L}_{\mathbf{A}}^{\frac\alpha2}}f)(x)
\end{equation}
and take the Fourier transform in time variable, we deduce that
\begin{equation}
\widehat{F^\epsilon}(\tau,x)=|\tau|^{1/2+\delta}\epsilon^{-1}(\hat{\rho}(\epsilon^{-1}(\tau-\mathcal{L}_{\mathbf{A}}^{\frac{\alpha}{2}}))f)(x)
:=\sum_{k=0}^\infty\widehat{F^\epsilon_k}(\tau,x),
\end{equation}
where
\begin{equation}
\widehat{F^\epsilon_k}(\tau,x)=|\tau|^{1/2+\delta}\epsilon^{-1}(\hat{\rho}(\epsilon^{-1}(\tau-\mathcal{L}_{\mathbf{A}}^{\frac{\alpha}{2}}))\circ\chi_{k\epsilon}^{\alpha}f)(x).
\end{equation}
Consequently, we have the following orthogonality
\begin{equation}
\int_{\mathbb{R}}F^\epsilon_k(t,x)\overline{F^\epsilon_\ell(t,x)}dt=(2\pi)^{-1}\int_{\mathbb{R}}\widehat{F^\epsilon_k}(\tau,x)\overline{
\widehat{F^\epsilon_\ell}(\tau,x)}d\tau=0,\quad |k-\ell|>100.
\end{equation}
Notice the fact that $f=\sum_{k=0}^\infty\chi^\epsilon_kf$, we obtain
\begin{align}
\|F^\epsilon(\cdot,x)\|_{L^2(\mathbb{R})}^2=\|\sum_{k=0}^\infty F^\epsilon_k(\cdot,x)\|_{L^2(\mathbb{R})}^2
&\lesssim\sum_{k=0}^\infty\|F^\epsilon_k(\cdot,x)\|_{L^2(\mathbb{R})}^2\nonumber\\
&=\sum_{k=0}^\infty\|\widehat{F^\epsilon_k}(\cdot,x)\|_{L^2(\mathbb{R})}^2.
\end{align}
Collecting the above facts, we deduce that% for $\alpha<1$
\begin{align}
\|\rho(\epsilon t)e^{it\mathcal{L}_{\mathbf{A}}^{\frac{\alpha}{2}}}f\|^2_{L^q(\mathbb{R}^2;L^\infty(\mathbb{R}))}
&\lesssim\sum_{k=0}^\infty\int_{\mathbb{R}}\|\widehat{F^\epsilon_k}(\tau,\cdot)\|_{L^q(\mathbb{R}^2)}^2d\tau\\
&=\epsilon^{-2}\sum_{k=0}^\infty\int_{(k-3)\epsilon}^{(k+3)\epsilon}|\tau|^{1+\delta}\|\hat{\rho}(\epsilon^{-1}
(\tau-\mathcal{L}_{\mathbf{A}}^{\frac{\alpha}{2}}))\chi_{k\epsilon}^{\alpha}f\|_{L^q(\mathbb{R}^2)}^2d\tau\nonumber\\
&\lesssim\epsilon^{-2}\sum_{k=0}^\infty\epsilon((k+1)\epsilon)^{1+\delta}
\|\chi_{k\epsilon}^{\alpha}f\|_{L^q(\mathbb{R}^2)}^2.\nonumber
\end{align}
%Recall \eqref{chi-k-epsilon-q}
%\begin{equation}
%\|\chi_{k\epsilon}^{\alpha}f\|_{L^q(\mathbb{R}^2)}^2
%\lesssim(((k+1)\epsilon)^\frac1\alpha-(k\epsilon)^\frac1\alpha)((k+1)\epsilon)^{(1-\frac{4}{q})/\alpha}
%\|\chi_{k\epsilon}^{\alpha}f\|_{L^2(\mathbb{R}^2)}^2,\quad k=0,1,2,\cdots.
%\end{equation}
Using \eqref{chi-k-epsilon-q} and letting $s:=1-\frac{2}{q}+\alpha\delta$, we further have
\begin{align}
\|\rho(\epsilon t)&e^{it\mathcal{L}_{\mathbf{A}}^{\frac{\alpha}{2}}}f\|_{L^q(\mathbb{R}^2;L^\infty(\mathbb{R}))}^2\\
&\lesssim\epsilon^{-2}\sum_{k=0}^\infty\epsilon((k+1)\epsilon)^{1+\delta}
(((k+1)\epsilon)^\frac1\alpha-(k\epsilon)^\frac1\alpha)((k+1)\epsilon)^{(1-\frac{4}{q})/\alpha}\|\chi_{k\epsilon}^{\alpha}f\|_{L^2(\mathbb{R}^2)}^2 \nonumber\\
&=\epsilon^{-2}\sum_{k=0}^\infty\epsilon^2((k+1)\epsilon)^{\frac{2}{\alpha}(1-\frac{2}{q})+\delta}
\|\chi_{k\epsilon}^{\alpha}f\|_{L^2(\mathbb{R}^2)}^2\nonumber\\
&=\sum_{k=0}^\infty\|((k+1)\epsilon)^s\chi_{k\epsilon}^{\alpha}f\|_{L^2(\mathbb{R}^2)}^2\thickapprox\|(\sqrt{\mathcal{L}_{\mathbf{A}}}+\epsilon I)^sf\|_{L^2(\mathbb{R}^2)}^2.\nonumber
\end{align}

\noindent\textbf{Acknowledgements}
The authors would like to thank Junyong Zhang for the encouragement and constructive suggestions.
They would also like to thank the anonymous referees for their helpful comments and suggestions.

\begin{center}

\end{center}

\end{document}